\documentclass[12pt,a4paper]{article}

\RequirePackage{etex}
\usepackage[utf8]{inputenc}
\usepackage{amsmath, amssymb, amsthm, amsfonts, mathrsfs,url}
\usepackage{bbm, commath}
\usepackage{multicol}
\usepackage{graphicx}
\usepackage{tikz-network}
\usetikzlibrary{positioning,patterns, calc}
\usetikzlibrary{graphs,graphs.standard,quotes}
\usepackage{authblk}
\usepackage{enumitem}

\newtheorem{thm}{Theorem}

\newtheorem{prop}{Proposition}

\newtheorem{cor}{Corollary}
\newtheorem{exm}{Example}

\newtheorem{rem}{Remark}

\def \Zl {{\mathbbm Z}}
\def \Ql {{\mathbbm Q}}
\def \Rl {{\mathbbm R}}
\def \Cl {{\mathbbm C}}
\def \e {{\bf e}}
\def \vl {{\bf v}}
\def \x {{\bf x}}

\def \o {{\bf 0}}

\newcommand{\ob}[1]{\left(#1\right)}
\newcommand{\cb}[1]{\left\lbrace #1\right\rbrace}
\newcommand{\tb}[1]{\left[#1\right]}

\newcommand{\up}[1]{\overset{#1}{\uplus}~}

\newcommand{\eig}[1]{\sigma\left(#1\right)}

\title{\it Laplacian quantum walks on blow-up graphs}

\author[1]{Hermie Monterde}
\author[2]{Hiranmoy Pal}
\author[1]{Steve Kirkland}%\thanks{Research supported by NSERC under grant number RGPIN-2025-05547.}}
\affil[1]{Department of Mathematics, University of Manitoba, Winnipeg, MB, Canada R3T 2N2}
\affil[2]{Department of Mathematics, National Institute of Technology Rourkela, India-769008}
\date{\today}

\begin{document}
\maketitle

\begin{center}
    \textit{Dedicated to Daniel Szyld on the occassion of his 70th birthday}
\end{center}

\begin{abstract}
This paper is a sequel to the work of Bhattacharjya et al.\ (J. Phys. A-Math. 57.33: 335303, \url{https://doi.org/10.1088/1751-8121/ad6653}) on quantum state transfer on blow-up graphs, where instead of the adjacency matrix, we take the Laplacian matrix as the time-independent Hamiltonian associated with a blow-up graph. We characterize strong cospectrality, periodicity, perfect state transfer (LPST) and pretty good state transfer (LPGST) on blow-up graphs. We present several constructions of blow-up graphs with LPST and produce new infinite families of regular graphs where each vertex is involved in LPST. We also determine LPST and LPGST in blow-ups of classes of trees. Finally, if $n\equiv 0$ (mod 4), then the blow-up of $n$ copies of a graph $G$ has no LPST, but we show that under certain conditions, the addition of an appropriate matching this blow-up graph results in LPST.\\


\end{abstract}

\section{Introduction}

Accurate transfer of quantum states is a key task in the setting of quantum computing. For this reason, there is considerable interest in constructing  networks of interacting qubits possessing desirable quantum state transfer properties. The idea of using a network of interacting qubits to transmit quantum information dates back to Bose \cite{bose}, who represented the network as an undirected graph wherein qubits  correspond to vertices and entangled pairs  correspond to edges. A time--independent Hamiltonian is specified (typically either the adjacency matrix or the Laplacian matrix, though other candidates for the Hamiltonian are available) and  the transfer of  quantum states is modelled by a quantum walk, which we now describe in further detail.

Throughout, we let $G$ be a simple (loopless) connected unweighted undirected graph with vertex set $V$. A \textit{(continuous-time) quantum walk} on $G$ with Hamiltonian $L$ is described by the unitary matrix
\begin{center}
    $U(t)=\exp{\ob{itL}}$
\end{center}
where $t\in\Rl$, $i=\sqrt{-1}$, and $L$ is the Laplacian matrix of $G$. Since $U(t)$ is unitary, for each $u \in V$ we have $\sum_{v\in V} | U(t)_{u,v}|^2=1.$ Consequently, the quantity $| U(t)_{u,v}|^2$ is interpreted as the probability that state $u$ is transmitted to state $v$ at time $t.$

We say that $G$ has \textit{ Laplacian perfect state transfer} (LPST) between vertices $u$ and $v$ if there is a time $\tau>0$ and some $\gamma\in\Cl$ such that
\begin{equation}
\label{uv}
U\ob{\tau}\mathbf{e}_u=\gamma\mathbf{e}_v.
\end{equation} 
In this case $| U(\tau)_{u,v}|^2=1$, so the quantum state initially at vertex $u$ is transmitted to vertex $v$ at time $\tau$ with probability one. If we also assume that $u=v$ in (\ref{uv}), then $u$ is said to be \textit{periodic} in $G$ at time $\tau$. A graph is \textit{periodic} if it is periodic at all vertices at the same time. The study of perfect state transfer was initiated by Bose \cite{bose} and Christandl et al. \cite{chr1}, who used the adjacency matrix as the Hamiltonian. However, Godsil showed that perfect state transfer in simple unweighted graphs is rare relative to the adjacency matrix \cite{god2}. A more general fact was shown by Godsil, Kirkland and Monterde which implies that LPST in simple unweighted graphs is also rare \cite{godsil2025perfect}. The rarity of perfect state transfer led to the notion of pretty good state transfer \cite{god1,vin}. A graph $G$ exhibits Laplacian \textit{pretty good state transfer} (LPGST) between vertices $u$ and $v$ if there is a sequence $\{\tau_k\}\subseteq\Rl$ and some $\gamma\in\Cl$ such that
\begin{equation}
\label{uv1}
\displaystyle\lim_{k\rightarrow \infty}U\ob{\tau_k}\mathbf{e}_u=\gamma\mathbf{e}_v.
\end{equation}
That is, $|U(t)_{u,v}|^2$ can be made arbitrarily close to 1 through appropriate choices of $t$. If we also assume that $u=v$ in (\ref{uv1}), then $G$ is said to be \textit{almost periodic} at $u$ relative to the sequence $\{\tau_k\}$. We say that $G$ is \textit{almost periodic} if there is a sequence $\{\tau_k\}\subseteq\Rl$ and some $\gamma\in\Cl$ such that
\begin{equation}
\label{uv2}
\lim\limits_{k\to\infty}U\ob{\tau_k}=\gamma I
\end{equation}
where $I$ is the identity matrix of appropriate order. The complex number $\gamma$ in (\ref{uv}), (\ref{uv1}) and (\ref{uv2}) is called a \textit{phase factor} and has modulus one because $U(t)$ is unitary for all $t\in\Rl$. %, it follows that $\|U(t)\mathbf{e}_u\|=\|\mathbf{e}_u\|$, where $\|\cdot\|$ is the complex Euclidean norm. This implies that the complex number $\gamma$ in (\ref{uv}), (\ref{uv1}) and (\ref{uv2}) has a unit modulus.

The study of graphs with useful state transfer properties includes the construction of new such families of graphs from old ones. For example,   
a \textit{blow-up} of $n$ copies of a graph $G$, denoted $\up{2}G$, is the graph obtained by replacing every vertex of $G$ by an independent set of size $n$, where the copies of two vertices in $G$ are adjacent in the blow-up if and only if the two vertices are adjacent in $G$. Motivated by the work of Ge at al.\ on perfect state transfer in lexicographic products \cite{ge}, Bhattacharjya et al.\ investigated quantum walks on blow-up graphs  relative to the adjacency matrix \cite{bhattacharjya2024quantum}. In this paper, we study quantum walks on blow-up graphs relative to the Laplacian matrix.

In Section \ref{sec:trans}, we derive the transition matrix of the quantum walk on a blow-up graph relative to the Laplacian matrix. This allows us in Sections \ref{sec:per}, \ref{sec:sc} and \ref{sec:pgst} to provide  characterizations of periodicity (Theorem \ref{per}), strong cospectrality (Theorem \ref{sc}), LPGST (Theorem \ref{th1}) and LPST (Theorem \ref{th2}) in a blow-up graph. In Sections \ref{sec:pgst}, \ref{sec:gp} and \ref{sec:join}, we present constructions of blow-up graphs with LPST using Hadamard diagonalizable graphs, Cartesian products, directs products, and joins (see Theorems \ref{cart}, \ref{dir} and \ref{joins}). This leads to infinite families of regular blow-up graphs such that each vertex is involved in LPST but the underlying graphs do not admit LPST (see Remark \ref{rem}). In Section \ref{sec:path}, we characterize LPGST in blow-ups of paths and double stars. In particular, we show that blow-ups of paths do not admit LPGST (Theorem \ref{path}), in contrast to the adjacency case where infinite families of blow-up graphs admit LPGST \cite[Theorem 8]{bhattacharjya2024quantum}. In Section \ref{sec:pert}, we show that if $n\equiv 0$ (mod 4), then under mild conditions, the addition of an appropriate matching in $\up{n}G$ results in LPST (see Theorem \ref{pt1}). Note that since $n\equiv 0$ (mod 4) in this case, we have $n\geq 3$, and so $\up{n}G$ does not admit LPST. We discuss open problems in Section \ref{sec:fw}.

There are similarities between the Laplacian and adjacency quantum walks on blow-up graphs, such as the absence of strong cospectrality in $\up{n}G$ for all $n\geq 3$. However, there are also  stark differences. Laplacian periodicity is preserved in a blow-up graph, but not adjacency periodicity. The characterizations of Laplacian and adjacency strong cospectrality in $\up{2}G$ also do not coincide unless $G$ is  regular (see Theorem \ref{sc}(2) and \cite[Theorem 2(2)]{bhattacharjya2024quantum}, respectively). Moreover, in the adjacency case, Theorem 4(3) (respectively, Theorem 3(3)) in \cite{bhattacharjya2024quantum} states that LPST (respectively, LPGST) between two copies of a vertex in $\up{2}G$ is equivalent to periodicity (respectively, almost periodicity) of the vertex in $G$ with the appropriate phase factor. However, for the Laplacian case, this is not enough; we need additional number-theoretic conditions on the degree of the vertex in $G$ for LPST or LPGST to occur in $\up{2}G$ (see Theorem \ref{th1}(3) and Theorem \ref{th2}(3), respectively). These observations reveal the fundamental differences in the behaviour of Laplacian and adjacency quantum walks on blow-up graphs.

\section{Transition matrix}\label{sec:trans}

A \textit{blow-up} of $n$ copies of $G$, denoted $\up{n}G$, is the graph with vertex set $\Zl_n\times V$, where vertices $(l,u)$ and $(m,v)$ are adjacent in $\up{n}G$ if and only if $u$ and $v$ are adjacent in $G$. Our goal in this section is to derive the spectral decomposition for the transition matrix of $\up{n}G$.

Let $J_n$ be the $n \times n$ all-ones  matrix. Considering the lexicographic ordering on $\Zl_n\times V(G),$ the adjacency matrix $A_n$ of $\up{n}G$ is given by $A_n=J_n\otimes A$, i.e. the Kronecker product of $J_n$ and $A$. The Laplacian matrix of $\up{n}G$ can be written as
\[L_n=nI_n\otimes D - J_n\otimes A,\]
where $D=\text{diag}\ob{d_1,\ldots,d_m}$ is the degree matrix of $G$. Here $L=D-A$ is the Laplacian matrix of $G$. Suppose
\[L=\sum_{j=1}^{r}\lambda_rE_r\]
is the spectral decomposition of $L$ (here $\lambda_j$ is an eigenvalue of $L$ with corresponding eigenprojection matrix $E_j, j=1, \ldots, r$). Note that
%where $\lambda_1,\lambda_2,\ldots,\lambda_r$ are the distinct eigenvalues of $L$ and $E_j$ is the idempotent corresponding to $\lambda_j.$
\[L_n\ob{\frac{1}{n}J_n\otimes E_j}=\ob{nI_n\otimes D - J_n\otimes A}\ob{\frac{1}{n}J_n\otimes E_j}=n\lambda_j\ob{\frac{1}{n}J_n\otimes E_r},\]
where $j\in\{1,2,\ldots,r\}$. Similarly 
\[L_n\ob{\ob{I_n-\frac{1}{n}J_n}\otimes \e_l\e_l^T}=nd_l\ob{\ob{I_n-\frac{1}{n}J_n}\otimes \e_l\e_l^T},\]
where $l\in\{1,2,\ldots,m\}$, and $\e_l$ is the $l$-th standard unit basis vector in  $\Cl^m.$ We now write the spectral decomposition of the Laplacian matrix $L_n$ of the blow-up $\up{n}G$ as 
\[L_n=\sum_{j=1}^{r}n\lambda_j\ob{\frac{1}{n}J_n\otimes E_j}+\sum_{l=1}^{m}nd_l\ob{I_n-\frac{1}{n}J_n}\otimes \e_l\e_l^T.\] 
This yields the transition matrix of the blow-up relative to the Laplacian:
\[U_n(t)=\sum_{j=1}^{r}\exp{\ob{-in\lambda_jt}}\ob{\frac{1}{n}J_n\otimes E_j}+\sum_{l=1}^{m}\exp{\ob{-ind_lt}}\ob{I_n-\frac{1}{n}J_n}\otimes \e_l\e_l^T.\]
Accordingly, for a vertex $u$ in $G$
\begin{eqnarray}\label{he1}
\e^T_{(0,u)}U_n(t)\e_{(1,u)} = \frac{1}{n}\sum_{j=1}^{r}\tb{\exp{\ob{-in\lambda_jt}} - \exp{\ob{-ind_ut}}}(E_j)_{u,u}.
\end{eqnarray}
For more about the eigenvalues and eigenvectors of the Laplacian of a blow-up, see \cite{de2016clique}.

We also make an observation about the vertices in $\up{n}G$. A pair of vertices $u$ and $v$ in a graph $G$ are called \textit{twins} if $N\ob{u}\setminus\cb{v}=N\ob{v}\setminus\cb{u}$. Twins that are not adjacent are called \textit{false twins}. A \textit{twin set} is a set of vertices in a graph whose elements are pairwise twins. For a vertex $u$ in $G$, $T_u:=\cb{(j,u):j\in\Zl_n}$ is the set of all copies of vertex $u$ in $\up{n}G$. Notice that $T_u$ is a set of twins in $\up{n}G$. Thus, for all $n\geq 2$, each vertex in $\up{n}G$ is contained in a twin set of size at least $n$. In fact, the following holds \cite[Proposition 1]{bhattacharjya2024quantum}.

\begin{prop}
\label{twins}
Let $u$ and $v$ be two vertices in $G$. The set $T_u\cup T_v$ is a twin set in $\up{n}G$ if and only if $u$ and $v$ are false twins in $G.$
\end{prop} 

\section{Periodicity}\label{sec:per}

Denote the set of distinct Laplacian eigenvalues of $G$ by $\sigma(G)$. The \textit{eigenvalue support} of vertex $v$ in $G$ is the set 
\begin{center}
$\sigma_v(G)=\cb{\lambda\in\eig{G}:E_{\lambda}\e_v\neq \o}.$
\end{center}
The following observation determines the eigenvalue support of a vertex in a blow-up.

\begin{prop}\label{pp1}
Let $n\geq 2$. Then $\sigma\left(\up{n}G\right)=n\cdot\sigma(G)\cup \{n\cdot d_u:u\in V(G)\}$. Moreover, for any vertex $u$ of $G$ and for any $j\in\Zl_n$, we have $\sigma_{(j,u)}\left(\up{n}G\right)=n\cdot\sigma_u(G)\cup\cb{n\cdot d_u}.$
\end{prop}

We say that vertex $u$ in $G$ is \textit{periodic} at time $\tau>0$ if $|U(\tau)_{u,u}|=1$. The minimum such $\tau>0$ is called the \textit{minimum period} of $u$, denoted by $\rho_u$. Since $0$ is an eigenvalue of $L$ having the all-ones vector as an associated eigenvector, $0\in\sigma_u(G)$. As $L$ is also positive semidefinite, \cite[Theorem 6.1]{god2} yields the following result.

\begin{thm}
\label{per}
Vertex $u$ in $G$ is periodic if and only if $\sigma_u(G)\subset\Zl$. In this case, $\rho_v=\frac{2\pi}{g}$, where $g=\operatorname{gcd}(\sigma_u(G)\backslash \{0\})$.
\end{thm}

The following is immediate from Theorem \ref{per}.

\begin{cor}\label{hc1}
$\up{n}G$ is periodic at vertex $(j,u)$ for some $n\geq 1$ if and only if $\sigma_u(G)\subset\Zl$. Moreover, if $\up{n}G$ is periodic at vertex $(j,u)$ for some $n\geq 1$, then $\up{n}G$ is periodic at vertex $(j,u)$ for all $n\geq 1$ and $j\in\Zl_n$ with $\rho_{(j,u)}=\frac{2\pi}{nh}$, where $h=\operatorname{gcd}(\sigma_u(G)\backslash \{0\}\cup \{d_v\})$.
\end{cor}

In \cite[Example 3]{bhattacharjya2024quantum}, adjacency periodicity of a vertex in a graph need not extend to a blow-graph. However, from Theorem \ref{per} and Corollary \ref{hc1}, it follows that blow-up graphs preserve Laplacian periodicity. This difference in behaviour between the adjacency and Laplacian dynamics has also been observed for join graphs \cite{kirkland2023quantum}.

A graph is \textit{periodic} if it is periodic at all vertices at the same time. The following result can be used to construct periodic graphs with small periods using the blow-up operation. Recall that a graph is \textit{Laplacian integral} if all eigenvalues of the corresponding Laplacian matrix are integers. 
\begin{cor}
$\up{n}G$ is periodic for some $n\geq 1$ if and only if $G$ is Laplacian integral. Moreover, if $\up{n}G$ is periodic for some $n\geq 1$, then $\up{n}G$ is periodic for all $n\geq 1$ with minimum period $\frac{2\pi}{nh}$, where $h=\operatorname{gcd}( \sigma(G)\backslash \{0\}\cup \{d_u:u\in V(G)\})$.
\end{cor}

\section{Strong cospectrality}\label{sec:sc}

Two vertices $u$ and $v$ in a graph $G$ are said to be \textit{strongly cospectral} if for every $\lambda\in\sigma_u(G)$,
\begin{equation*}
E_{\lambda}\textbf{e}_u=\pm E_{\lambda}\textbf{e}_v.
\end{equation*} 
If $u$ and $v$ are strongly cospectral, then we may write $\sigma_u(G)=\sigma_{uv}^+(G)\cup \sigma_{uv}^-(G)$, where
\begin{equation*}
\sigma_{uv}^+(G)=\{\lambda:E_{\lambda}\textbf{e}_u= E_{\lambda}\textbf{e}_v\neq \textbf{0}\}\quad \text{and}\quad \sigma_{uv}^-(G)=\{\lambda:E_{\lambda}\textbf{e}_u=- E_{\lambda}\textbf{e}_v\neq \textbf{0}\}.
\end{equation*} 
Strong cospectrality is a necessary condition for LPGST \cite[Lemma 13.1]{god1}. Thus, in order to characterize LPGST in blow-up graphs, we first need to characterize strong cospectrality.

\begin{thm}
\label{sc}
Let $G$ be a graph.
\begin{enumerate}
\item If $n\geq 3$, then $\up{n}G$ does not exhibit strong cospectrality.
\item Let $n=2$. If $u$ is a vertex in $G$ with degree $d_u$, then vertices $(0,u)$ and $(1,u)$ are strongly cospectral in $\up{2}G$ if and only if $d_u\notin\sigma_u(G)$, in which case
\begin{equation}
\label{+-}
\sigma_{(0,u),(1,u)}^+(G)=2\cdot \sigma_{u}(G)\quad \text{and}\quad \sigma_{(0,u),(1,u)}^-(G)=\{2d_u\}.
\end{equation}
Moreover, vertex $(0,u)$ can only be strongly cospectral with $(1,u)$ in $\up{2}G$.
\end{enumerate}
\end{thm}

\begin{proof}
We prove 1. For each $u\in V(G)$, let $T_u=\{(j,u):j\in\Zl_n\}$, which is a twin set in $\up{n}G$. If $n\geq 3$, then $|T_u|\geq 3$, and so \cite[Corollary 3.10]{mon1} yields the desired conclusion. To prove 2, let $n=2$. Note that if $\lambda_j\in\sigma_u(G)$, then the idempotent $\frac{1}{2}J_2\otimes E_j$ associated with each $2\lambda_j\in\sigma_u(\up{2}G)$ satisfies $(\frac{1}{2}J_2\otimes E_j)(\textbf{e}_0\otimes \textbf{e}_u)=(\frac{1}{2}J_2\otimes E_j)(\textbf{e}_1\otimes \textbf{e}_u)$. Now, if $d_u\notin \sigma_u(G)$, then the idempotent associated with $2d_u\in \sigma_u(\up{2}G)$ is given by $\ob{I_2-\frac{1}{2}J_2}\otimes \e_u\e_u^T$, and so we get $\left(\ob{I_2-\frac{1}{2}J_2}\otimes \e_u\e_u^T\right)(\textbf{e}_0\otimes \textbf{e}_u)=-\left(\ob{I_2-\frac{1}{2}J_2}\otimes \e_u\e_u^T\right)(\textbf{e}_1\otimes \textbf{e}_u)$. In other words, $(0,u)$ and $(1,u)$ are strongly cospectral in $\up{2}G$ and (\ref{+-}) holds. On the other hand, if $\lambda_j=d_u\in \sigma_u(G)$ for some $j$, then the idempotent associated with $2d_u\in \sigma_u(\up{2}G)$ is given by $E_{2d_u}=\frac{1}{2}J_2\otimes E_j+\left(\ob{I_2-\frac{1}{2}J_2}\otimes \e_u\e_u^T\right)$, and so $E_{2d_u}(\textbf{e}_0\otimes \textbf{e}_u)\neq \pm E_{2d_u}(\textbf{e}_1\otimes \textbf{e}_u)$. This proves the first statement of (2). The latter follows from \cite[Theorem 3.9(2)]{mon1}.
\end{proof}

If $n\geq 3$ then the absolute value of the expression on the right of \eqref{he1} is bounded above by $\frac{2}{3}.$ Hence, there is no LPGST in the blow-up $\up{n}G$ whenever $n\geq 3$, a fact that coincides with Theorem \ref{sc}(1). Thus, we consider the case $n=2$ for the discussion on LPGST.

We now apply Theorem \ref{sc}(2) to regular graphs. As usual, we denote the complete graph on $n$ vertices, the $n$-path, the $n$-cycle, and the $d$-cube, by $K_n, P_n C_n$ and $Q_d$, respectively. 

\begin{cor}  
\label{reg}
Let $G$ be a $d$-regular graph with vertex $u$. Vertices $(0,u)$ and $(1,u)$ are strongly cospectral in $\up{2}G$ if and only if $d\notin\sigma_u(G)$. In particular, if $u$ is any vertex, then vertices $(0,u)$ and $(1,u)$ are strongly cospectral in: (i)
$\up{2}K_{d+1}$ for all $d\geq 1$, (ii) $\up{2}C_n$ if and only if $n\not\equiv 0$ (mod 4), and (iii) $\up{2}Q_d$ if and only if $d$ is odd.
\end{cor}

\begin{proof}
The first statement follows from Theorem \ref{sc}(2). For $G\in\{K_{d+1}, C_n,Q_d\}$, we have $\sigma_u(G)=\sigma(G)$ for any vertex $u$ of $G$. Since $\sigma(K_{d+1})=\{0,d+1\}$,  $\sigma(C_{m})=\{2(1-\cos(\frac{2j\pi}{n})):0\leq j\leq \lfloor\frac{n}{2}\rfloor\}$, and $\sigma(Q_d)=\{2j:0\leq j\leq d\}$, the second statement follows from the first and the fact that $K_{d+1}$ and $Q_d$ are $d$-regular and $C_n$ is 2-regular.   
\end{proof}

\section{Pretty good state transfer}\label{sec:pgst}

We now characterize Laplacian pretty good state transfer (LPGST) in blow-up graphs. Since strong cospectrality is required for LPGST \cite[Lemma 13.1]{god1}, by virtue of Theorem \ref{sc}(2), it suffices to consider vertices $(0,u)$ and $(1,u)$ in $\up{n}G$ where $n=2$ and $d_u\notin\sigma_u(G)$. Moreover, since $0$ is an eigenvalue of the Laplacian matrix of a graph with the all-ones vector as an associated eigenvector, it follows that $0$ is in the eigenvalue support of every vertex in a graph. From the spectral decomposition of $L$, it follows that the phase factor for Laplacian periodicity and LPGST is always $\gamma=1$.

\begin{thm}\label{th1} 
Let $u$ be a vertex in $G$ with degree $d_u$ such that $d_u\notin\sigma_u(G)$. The following are equivalent.
\begin{enumerate}
    %\item $\up{2}G$ exhibits LPGST between $(0,u)$ and $(1,u)$.
    \item $\up{2}G$ exhibits LPGST between $(0,u)$ and $(1,u)$ with phase factor $\gamma=1.$
    \item There exists a sequence $\{\tau_k\}$ such that $\lim\limits_{k\to\infty}\exp{\left(-2id_u \tau_k\right)}=-1$ and \[\lim\limits_{k\to\infty}\exp{\left(-2i\lambda \tau_k\right)}= 1 \text{ for each } \lambda\in\sigma_u(G).\] 
 \item $G$ is almost periodic at $u$ relative to the sequence $\{\tau_k\}$ with phase factor $\gamma=1$ and 
 \[\lim\limits_{k\to\infty}\exp{\left(-2id_u \tau_k\right)}=-1.\]
 \item If $m_j$ are integers such that $\hspace{-0.1in}\displaystyle\sum_{\lambda_j\in\sigma_u(G)}\hspace{-0.1in}m_j(\lambda_j-d_u)=0$, then $\hspace{-0.1in} \displaystyle\sum_{\lambda_j\in\sigma_u(G)}\hspace{-0.1in} m_j$ is even.
\end{enumerate}
\end{thm}
 
The proof of Theorem \ref{th1} is similar to that of \cite[Theorem 3]{bhattacharjya2024quantum}, and so we omit it. We note however that the equivalence of 1 and 4 above  follows from \cite[Theorem 13]{Kirk2}.

Theorem \ref{th1} now yields a characterization of LPST in blow-up graphs. For a non-zero integer $n$, the \textit{2-adic norm} of $n$, denoted $\nu_2(n)$, is the exponent of the largest power of $2$ that divides $n.$ Further, we adopt the usual convention  $\nu_2(0) = \infty$. In particular, $\nu_2(0)>\nu_2(n)$ for any nonzero integer $n.$ 

\begin{thm}\label{th2}
Let $u$ be a vertex in $G$ with degree $d_u$ such that $d_u\notin\sigma_u(G)$. The following are equivalent.
\begin{enumerate}
    \item There exists $\tau>0$ such that  $\up{2}G$ exhibits LPST at time $\tau$ between $(0,u)$ and $(1,u)$ with phase factor $\gamma=1.$
    \item There exists $\tau>0$ such that $-\exp{\left(-2id_u \tau\right)}=\exp{\left(-2i\lambda \tau\right)}= 1$
    %$\exp{\left(-2id_u \tau\right)}= -1$ and $\exp{\left(-2i\lambda \tau\right)}= 1 \text{ .
    for each  $\lambda\in\sigma_u(G)$.
 \item There exists $\tau>0$ such that $G$ is periodic at $u$ at time $\tau$ with phase factor $\gamma=1$ and $\exp{\left(-2id_u \tau\right)}= -1.$
 \item $\sigma_u(G)\subset\Zl$ and $\nu_2(d_u) < \nu_2(\lambda)$ for all $\lambda\in\sigma_u(G)\backslash\{0\}$.
\end{enumerate}
Moreover, when one of statements 1-4 holds, the minimum LPST time is $\tau = \frac{\pi}{2h}$, where $h=\operatorname{gcd}(\sigma_u(G)\backslash \{0\}\cup \{d_v\})$.
\end{thm}

We comment that the equivalence of statements 1 and 4 above follows from \cite[Theorem 10]{Kirk2}, while the minimum LPST time follows from Corollary \ref{hc1} and the fact that the minimum LPST time is half the minimum period \cite{god1}.

Our next result is immediate from Theorem \ref{th2}(4).

\begin{cor}
\label{d}
Let $G$ be a graph with vertex $u$ such that $d_u$ is odd. LPST occurs in $\up{2}G$ between vertices $(0,u)$ and $(1,u)$ if and only if $\sigma_u(G)$ consists of even integers.
\end{cor}

\begin{rem}
\label{rem}
The above corollary applies in particular to regular graphs with odd degree whose eigenvalues are all even integers.
\end{rem}

In \cite[Section 7]{bhattacharjya2024quantum}, LPGST in complete graphs and cycles was characterized. Since these graphs are regular, the same results apply to the Laplacian case. Our next endeavour is to find %In what follows, we supply 
an infinite family of regular graphs whose blow-ups admit LPST.

Recall that a graph $G$ is \textit{Hadamard diagonalizable} if its Laplacian matrix is diagonalizable by a Hadamard matrix. The LPST properties of Hadamard diagonalizable graphs were investigated in \cite{JOHNSTON2017375}. It is known that a Hadamard diagonalizable graph $G$ on $n$ vertices is $d$-regular for some $d$, all eigenvalues in $\sigma(G)$ are even integers and $n\equiv 0$ (mod 4) \cite{Barik2011}. Thus, Corollary \ref{d} and Remark \ref{rem} yield the following result.

\begin{cor}
\label{hd}
If $G$ is a $d$-regular Hadamard diagonalizable graph for some odd $d$, then LPST occurs in $\up{2}G$ between $(0,u)$ and $(1,u)$ for any vertex $u$ of $G$ at $\frac{\pi}{2}$.
\end{cor}

\section{Graph products}\label{sec:gp}

We now construct infinite families of blow-ups in which each vertex is involved in LPST. 
%Denote the Kronecker product of matrices $A_1$ and $A_2$ by $A_1\otimes A_2$.

Let $G$ and $H$ be graphs on $m$ and $n$ vertices respectively. The \textit{Cartesian product} of $G$ and $H$, denoted $G\square H$, is the graph with Laplacian matrix
\begin{center}
$L(G\square H)=L(G)\otimes I_n+I_m\otimes L(H)$.
\end{center}

The \textit{direct product} of $G$ and $H$, denoted $G\times H$, is the graph with adjacency matrix
\begin{center}
$A(G\otimes H)=A(G)\otimes A(H)$.
\end{center}

\begin{thm}
\label{cart}
Let $G=G_1\square\cdots\square G_d$, where each $G_j$ is a $k_j$-regular and $k=k_1+\ldots+k_d$ is odd. If $\sigma(G_j)$ consists of even integers for each $j\in\{1,\ldots,d\}$, then LPST occurs in $\up{2}G$ between $(0,u)$ and $(1,u)$ for any vertex $u$ of $G$ at $\frac{\pi}{2}$.
\end{thm}

\begin{proof}
First, note that $G$ is $k$-regular. The Laplacian eigenvalues of the Cartesian product of $X$ and $Y$ are all of the form $\lambda+\mu$, where $\lambda\in\sigma(X)$ and $\lambda\in\sigma(Y)$. Since $\sigma(G_{j})$ consists of even integers for each $j\in\{1,\ldots,d\}$, applying the preceding fact $d-1$ times to $G$ starting with $G_{1}$, we get that the eigenvalues in $\sigma(G)$ are all even. Since $k$ is odd, invoking Corollary \ref{d} and Remark \ref{rem} yield the desired result.
\end{proof}

The \textit{Hamming graph}, denoted $H(d,q)$, is the Cartesian product of $d$ copies of $K_q$.

\begin{cor}
\label{cart1}
If $G=K_{n_1}\square\cdots\square K_{n_d}$, where each $n_j\geq 2$ is even and $d\geq 2$ is odd, then LPST occurs in $\up{2}G$ between $(0,u)$ and $(1,u)$ for any vertex $u$ of $G$. In particular, if $q$ is even, then LPST occurs in $\up{2}H(d,q)$ between $(0,u)$ and $(1,u)$ for any vertex $u$ of $H(d,q)$.
\end{cor}

\begin{proof}
Since $\sigma(K_{n_j})=\{0,n_j\}$ consists of even integers for all $j\in\{1,\ldots,d\}$, letting $G_j=K_{n_j}$ and $k_j=n_{j}-1$ in Theorem \ref{cart} yields the first statement. The second statement follows from the first by taking $n_1=\cdots=n_d=q$ to be even.
\end{proof}

Since $H(2,d)=Q_d$, the second statement of Corollary \ref{cart1} together with Corollary \ref{reg}(iii) implies that $\up{2}Q_d$ admits LPST if and only if $d$ is odd. In this case, LPST occurs between vertices $(0,u)$ and $(1,u)$ for any vertex $u$ of $Q_d$ at time $\frac{\pi}{2}$.

We now use the direct product to construct infinite families of blow-ups with LPST.

\begin{thm}
\label{dir}
Let $G=G_{1}\times\cdots\times G_{d}$ such that each $G_j$ is $k_j$-regular where each $k_j\geq 1$. % and $d_j\notin \sigma(G_j)$ for each $j$. 
If for each $j$, the adjacency eigenvalues of $G_j$ have equal 2-adic norms, then LPST occurs in $\up{2}G$ between $(0,u)$ and $(1,u)$ for any vertex $u$ of $G$.
\end{thm}

\begin{proof}
Our assumption that the adjacency eigenvalues of $G_j$ for each $j$ consists of integers with equal 2-adic norms implies that $0$ is not an adjacency eigenvalue of any $G_j$. Now, the adjacency eigenvalues of $X\times Y$ are all of the form $\lambda\mu$, where $\lambda\in\sigma(X)$ and $\lambda\in\sigma(Y)$. Combining this fact with our assumption implies that the adjacency eigenvalues $\theta_{r}$ of $G$ have equal 2-adic norms. Thus, $\theta_r\neq 0$ for all $r$. Because $G$ is $k$-regular, where $k=\prod_jk_j$, each Laplacian eigenvalue $\lambda_r\in\sigma(G)$ satisfies $\lambda_r:=k-\theta_r\neq k$ for all $r$. %Thus, $\theta_r=k-\lambda_r$ and
Since the $\theta_{r}$'s have equal 2-adic norms, we have %the $\theta_r$'s have equal 2-adic norms. This implies that 
$\nu_2(k)=\nu_2(k-\lambda_r)=\nu_2(k-\lambda_s)$ for all $\lambda_r,\lambda_s\in\sigma(G)$. Equivalently, $\nu_2(k)<\nu_2(\lambda_r)$. Invoking Theorem \ref{th2}(4) yields the desired conclusion. 
\end{proof}

\begin{cor}
\label{dir1}
If $G=K_{n_1}\times\cdots\times K_{n_d}$, where each $n_j\geq 2$ is even, then LPST occurs in $\up{2}G$ between $(0,u)$ and $(1,u)$ for any vertex $u$ of $G$.
\end{cor}

\begin{proof}
Since each $n_j$ is even, the adjacency eigenvalues $-1$ and $n_j-1$ of $K_{n_j}$ are both odd. Applying Theorem \ref{dir} with $G_j=K_{n_j}$ for each $j$ yields the result.
\end{proof}

For a non-bipartite graph $H$, the graph $K_2\times H$ is known as the \textit{bipartite double} of $H$. The following is immediate from Theorem \ref{dir}.

\begin{cor}
\label{dir2}
Let $H$ be a regular non-bipartite graph whose adjacency eigenvalues are integers with equal $\nu_2$ values. If $G=K_{2}\times H$, then LPST occurs in $\up{2}G$ between $(0,u)$ and $(1,u)$ for any vertex $u$ of $G$.
\end{cor}

In Corollary \ref{dir2}, taking $H=K_n$ for even $n$ yields an infinite family of blow-ups of bipartite doubles where each vertex is involved in LPST.

\begin{rem}
Let $n_j\geq 3$ be even and $X\in\{G,H\}$, where $G=K_{n_1}\square\cdots\square K_{n_d}$ for odd $d\geq 2$ and $H=K_{n_1}\times\cdots\times K_{n_d}$ with $\prod_{j=1}^d(n_j-1)\neq \frac{1}{2}\prod_{j=1}^dn_j$. By Theorem 6.2 in \cite{sed} and Theorem 34(2) \cite{monterde2023new} respectively, $G$ and $H$ are sedentary. That is, for each vertex $u$ of $X$,  $|U(t)_{u,u}|$ is bounded away from 0 for all $t$. For this reason, $X$ do not admit LPST. Invoking Corollaries \ref{cart1} and \ref{dir1} respectively, we get that $\up{2}X$ admits LPST between $(0,u)$ and $(1,u)$ for any vertex $u$ of $X$. This yields infinite families of regular blow-up graphs where each vertex is involved in LPST, but the underlying graphs do not admit LPST.
\end{rem}

\section{Joins}\label{sec:join}

Let $G$ and $H$ be graphs on $m$ and $n$ vertices, respectively. The \textit{join} of $G$ and $H$, denoted $G\vee H$, is the graph obtained by adding all edges between $G$ and $H$ \cite{kirkland2023quantum}. Note that $\up{2}(G\vee H)=(\up{2}G)\vee (\up{2} H)$, and so the blow-up of a join is simply the join of the blow-ups of the underlying graphs.

\begin{thm}
\label{joins}
Let $u$ be a vertex of $G$ such that $d_u\notin\sigma_u(G)$. Then 
LPST occurs in $\up{2}(G\vee H)$ between $(0,u)$ and $(1,u)$ if and only if $\sigma_u(G)\subset\Zl$ and one of the following conditions hold. 
\begin{enumerate}
\item $G$ is connected, $\nu_2(m+n)>\nu_2(d_u+n)$ and $\nu_2(\lambda+n)>\nu_2(d_u+n)$ for all $\lambda\in\sigma_u(G)\backslash \{0\}$.
\item $G$ is disconnected, $u$ is not an isolated vertex in $G$, $\nu_2(n)>\nu_2(d_u+n)$, $\nu_2(m+n)>\nu_2(d_u+n)$ and $\nu_2(\lambda+n)>\nu_2(d_u+n)$ for all $\lambda\in\sigma_u(G)\backslash \{0\}$.
\end{enumerate}
Moreover, the minimum LPST time is $\tau = \frac{\pi}{2h}$, where $h=\operatorname{gcd}(S)$, where $S=\{\lambda+n:\lambda\in\sigma_u(G)\backslash \{0\}\cup \{d_u,m\}\}$ if 1 holds, and $S=\{\lambda+n:\lambda\in\sigma_u(G)\cup \{d_u,m\}\}$ otherwise.

\end{thm}
\begin{proof}
Let $u$ be a vertex of $G$ such that $d_u\notin\sigma_u(G)$. Define $S=\{\lambda+n:\lambda\in\sigma_u(G)\backslash\{0\}\}$. By \cite[Lemma 3]{kirkland2023quantum}, we have $\sigma_{u}(G\vee H)=S\cup\{0,m+n\}$ whenever $G$ is connected, and $\sigma_{u}(G\vee H)=S\cup\{0,m+n,n\}$ otherwise. Now, the degree of vertex $u$ in $G\vee H$ is $d_u+n$. First, suppose $G$ is connected. If $G\neq K_1$, then $0<d_u<m$ and so $n<d_u+n<m+n$. Since $d_u\notin\sigma_u(G)$, we get that $d_u+n\notin \sigma_{u}(G\vee H)$. If $G=K_1$ (which is connected), then $d_u+n=n\notin \sigma_{u}(G\vee H)=\{0,m+n\}$. Applying Theorem \ref{th2}(4) to both cases yields the desired result in 1. Now, suppose $G$ is disconnected. If $u$ is not isolated, then $d_u+n\notin \sigma_{u}(G\vee H)$, and so Theorem \ref{th2}(4) yields the desired result in 2. On the other hand, if $u$ is isolated, then $d_u+n=n\in \sigma_u(G\vee H)$. By Theorem \ref{sc}(2), vertices $(0,u)$ and $(1,u)$ are not strongly cospectral in $\up{2}(G\vee H)$, and so they cannot admit LPST. Therefore, $u$ must not be an isolated vertex in $G$ if $(0,u)$ were to admit LPST in $\up{2}(G\vee H)$.
\end{proof}

\begin{cor}
\label{joincor}
The following hold.
\begin{enumerate}
\item If $G$ is connected and $\sigma(G)\cup\{m,n\}\subset \Zl$ consists of odd integers, then LPST occurs in $\up{2}(G\vee H)$ between $(0,u)$ and $(1,u)$ if and only if $d_u$ is even.
\item If $\sigma(G)\cup\{m,n\}\subset \Zl$ consists of  even integers, then LPST occurs in $\up{2}(G\vee H)$ between $(0,u)$ and $(1,u)$ if and only if $d_u$ is  odd. %This also holds if $G$ is disconnected.
\end{enumerate}
\end{cor}

\begin{proof}
The first statements in 1 and 2 are immediate from Theorem \ref{joins}(1), while the second statement in 2 is immediate from Theorem \ref{joins}(2). Now, if the second statement in 1 holds, then largest power of two dividing $n$ is bigger than that of $d_u+n$ by Theorem \ref{joins}(2), which is a contradiction because both $n$ and $d_u+n$ are odd in this case. Thus, 1 and 2 hold.
\end{proof}

For a Hadamard diagonalizable graph $G$, $\sigma(G)\cup\{m\}$ consists of even integers. Thus, Corollary \ref{joincor} gives us the following result.

\begin{cor}
If $G$ is a $d$-regular Hadamard diagonalizable graph, then LPST occurs in $\up{2}(G\vee H)$ between $(0,u)$ and $(1,u)$ if and only if $d$ is  odd and $n$ is even. In this case, LPST occurs between $(0,u)$ and $(1,u)$ for any vertex $u$ of $G$.
\end{cor}

In the next example, $G\cup H$ denotes the disjoint union of two graphs $G$ and $H$.

\begin{exm}
Consider the graph $G=O_1\vee (K_2\cup K_2)$. Then $\sigma(G)=\{0,3,5\}$ and each vertex of $G$ has even degree. Since $G$ is connected, Corollary \ref{joincor} implies that for all odd $n$, LPST occurs in $\up{2}(G\vee H)$ between $(0,u)$ and $(1,u)$ for any vertex $u$ of $G$. In particular, if $H=O_1$, then $\up{2}(G\vee H)$ admits LPST (see Figure \ref{ex}).
\end{exm}

If $G=K_1$ in Theorem \ref{joins}(1) and $V(K_1)=\{u\}$, then $m=1$ and $\sigma_{u}(G\vee H)=\{0,n+1\}$. Since  $\nu_2(n+1)> \nu_2( d_u+n)=\nu_2(n)$ if and only if $n$ is odd, we get the following result for blow-ups of cones.

\begin{figure}[h]
\centering
\begin{tikzpicture}[scale=.3,auto=left]
%\tikzstyle{every node}=[draw, circle, inner sep=0pt, minimum size=.22cm]
\tikzstyle{every node}=[draw, circle, thick, fill=black!0, scale=0.24]
 \node (0) at (-2,-1){\LARGE $\ a\ $};
 \node (1) at (-2,3){\LARGE $\ b\ $};
 \node (2) at (0.5,-1){\LARGE $\ c\ $};
 \node (3) at (5,-1){\LARGE $\ d\ $};
 \node (4) at (5,3){\LARGE $\ e\ $}; 
 \node (5) at (2.5,3){\LARGE $\ f\ $}; 

 \node (10) at (14.6,-2.4){\LARGE$(0,a)$};
 \node (11) at (14.6,4.4){\LARGE$(0,b)$};
 \node (12) at (19.8,-2.4){\LARGE$(0,c)$};
 \node (13) at (27.3,-2.4){\LARGE$(0,d)$};
 \node (14) at (27.3,4.4){\LARGE$(0,e)$}; 
 \node (15) at (22.15,4.4){\LARGE$(0,f)$}; 
 \node (101) at (10,-4.5){\LARGE$(1,a)$};
 \node (111) at (10,6.3){\LARGE$(1,b)$};
 \node (121) at (20.7,-6.2){\LARGE$(1,c)$};
 \node (131) at (31.5,-4.5){\LARGE$(1,d)$};
 \node (141) at (31.5,6.3){\LARGE$(1,e)$}; 
 \node (151) at (21.3,7.7){\LARGE$(1,f)$};

\draw[thick,black!70] (0)--(2);
\draw[thick,black!70] (0)--(1);
\draw[thick,black!70] (1)--(2);
\draw[thick,black!70] (3)--(2);
\draw[thick,black!70] (3)--(4);
\draw[thick,black!70] (4)--(2);
\draw[thick,black!70] (5)--(0);
\draw[thick,black!70] (5)--(1);
\draw[thick,black!70] (5)--(3);
\draw[thick,black!70] (5)--(4);
\draw[thick,black!70] (5)--(2);

\draw[thick,black!70] (10)--(12);
\draw[thick,black!70] (10)--(11);
\draw[thick,black!70] (11)--(12);
\draw[thick,black!70] (13)--(12);
\draw[thick,black!70] (13)--(14);
\draw[thick,black!70] (14)--(12);
\draw[thick,black!70] (15)--(10);
\draw[thick,black!70] (15)--(11);
\draw[thick,black!70] (15)--(13);
\draw[thick,black!70] (15)--(14);
\draw[thick,black!70] (15)--(12);

\draw[thick,black!70] (101)--(121);
\draw[thick,black!70] (101)--(111);
\draw[thick,black!70] (111)--(121);
\draw[thick,black!70] (131)--(121);
\draw[thick,black!70] (131)--(141);
\draw[thick,black!70] (141)--(121);
\draw[thick,black!70] (151)--(101);
\draw[thick,black!70] (151)--(111);
\draw[thick,black!70] (151)--(131);
\draw[thick,black!70] (151)--(141);
\draw[thick,black!70] (151)--(121);

\draw[thick,black!70] (10)--(121);
\draw[thick,black!70] (10)--(111);
\draw[thick,black!70] (11)--(121);
\draw[thick,black!70] (13)--(121);
\draw[thick,black!70] (13)--(141);
\draw[thick,black!70] (14)--(121);
\draw[thick,black!70] (15)--(101);
\draw[thick,black!70] (15)--(111);
\draw[thick,black!70] (15)--(131);
\draw[thick,black!70] (15)--(141);
\draw[thick,black!70] (15)--(121);

\draw[thick,black!70] (101)--(12);
\draw[thick,black!70] (101)--(11);
\draw[thick,black!70] (111)--(12);
\draw[thick,black!70] (131)--(12);
\draw[thick,black!70] (131)--(14);
\draw[thick,black!70] (141)--(12);
\draw[thick,black!70] (151)--(10);
\draw[thick,black!70] (151)--(11);
\draw[thick,black!70] (151)--(13);
\draw[thick,black!70] (151)--(14);
\draw[thick,black!70] (151)--(12);

 \end{tikzpicture}
 \caption{The graph $G\vee H$ where $G=O_1\vee (K_2\cup K_2)$ and $H=O_1$ (left); the graph $\up{2}(G\vee H)$ with LPST between vertices $(0,u)$ and $(1,u)$ for any vertex $u$ of $G\vee H$ (right)}
 \label{ex}
 \end{figure}
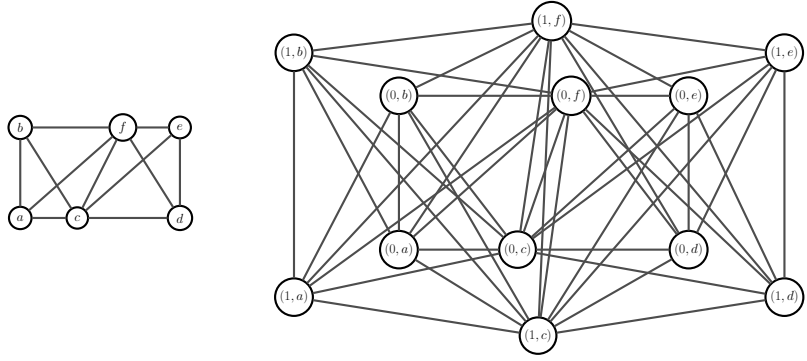
 
\begin{cor}
\label{cone}
If $G=K_{1}\vee H$ and $V(K_1)=\{u\}$, then LPST occurs in $\up{2}G$ at $\frac{\pi}{2}$  between $(0,u)$ and $(1,u)$ if and only if $n$ odd.
\end{cor}

Letting $H=O_n$ in Corollary \ref{cone} yields the following result.

\begin{cor}
\label{star}
$\up{2}K_{1,n}$ has LPST between vertices of degree $2m$ if and only if $n$ is odd.
\end{cor}

\section{Paths}\label{sec:path}

A landmark result of Coutinho and Liu states that $K_2$ and $P_3$ are the only trees that admit Laplacian perfect state transfer \cite{cou0}. Despite this fact, Corollary \ref{star} implies that there are infinite families of trees whose blow-ups admit LPST. This fact motivates us to further examine pretty good state transfer on blow-ups of trees, starting with paths in this section. 

The Laplacian eigenvalues of $P_n$ are
\begin{center}
$\theta_j=2\left(1-\cos\left(\frac{j\pi}{n}\right)\right)\quad j=0,1,\ldots,n-1$
\end{center}
Moreover, for every vertex $u$ of $P_n$, \cite[Lemma 4.4.1]{Bommel2019} yields
\begin{equation}
\label{supp}
    \sigma_u(P_n)=\{0\}\cup\{\theta_j:2n\nmid (2u-1)j\}
\end{equation}

\begin{thm}
\label{path}
$\up{2}P_n$ does not admit LPGST.
\end{thm}

\begin{proof}
Let $u$ be a vertex of $P_n$. First, let $u\in\{1,n\}$. Then $d_u=1$ and $\sigma_u(P_n)=\sigma(P_n)$. Suppose $n$ is even so that $\theta_{\frac{n}{2}}=2\in \sigma(P_n)$. Letting $m_1=m_{n-1}=1=-m_{\frac{n}{2}}$ and $m_j=0$ otherwise, we get that $\sum_{j=1}^{n-1}m_j=1$ is odd, but
\begin{center}
$\displaystyle\sum_{j=0}^{n-1}m_j(\theta_j-1)=(\theta_1-1)+(\theta_{n-1}-1)-\theta_{\frac{n}{2}}%=(1-\cos(\pi/n))+(1-\cos((n-1)\pi/n))-2
=0$.
\end{center}
Invoking Theorem \ref{th1}(4), we conclude that LPGST does not occur between $(0,u)$ and $(1,u)$ in $\up{2}P_n$. Now, suppose $n$ is odd. In this case, \cite[Lemma 4.3.2]{Bommel2019} states that
\begin{equation*}
\sum_{j=0}^{n-1}(-1)^{j}\cos\left(\frac{j\pi}{n}\right)=0.
\end{equation*}
Since $\cos(\frac{j\pi}{n})=\cos(\frac{(n-j)\pi}{n})$, the above expression gives us
\begin{equation}
\label{eqq}
\sum_{j=1}^{\frac{n-1}{2}}(-1)^{j-1}\cos\left(\frac{j\pi}{n}\right)=\frac{1}{2}.
\end{equation}
If $\frac{n-1}{2}$ is odd, then taking $m_j=(-1)^{j-1}$ for $j\in\{1,\ldots,\frac{n-1}{2}\}$ and $m_j=0$ otherwise yields $\sum_{j=1}^{\frac{n-1}{2}}m_j=\sum_{j=1}^{\frac{n-1}{2}}(-1)^{j-1}=1$ is odd, and making use of (\ref{eqq}) gives us
\begin{equation*}
\begin{split}
\sum_{j=1}^{\frac{n-1}{2}}m_j(\theta_j-1)&=\displaystyle\sum_{j=1}^{\frac{n-1}{2}}(-1)^{j-1}\left(1-2\cos\left(\frac{j\pi}{n}\right)\right)\\
&=\displaystyle\sum_{j=1}^{\frac{n-1}{2}}(-1)^{j-1}
-2\displaystyle\sum_{j=1}^{\frac{n-1}{2}}(-1)^{j-1}\cos\left(\frac{j\pi}{n}\right)=1-2(1/2)=0.
\end{split}
\end{equation*}
If $\frac{n-1}{2}$ is even, then taking $m_0=-1$, $m_j=(-1)^{j-1}$ for $j\in\{1,\ldots,\frac{n-1}{2}\}$ and $m_j=0$ otherwise yields $\sum_{j=0}^{\frac{n-1}{2}}m_j=-1$ is odd, and making use of (\ref{eqq}) give us
\begin{equation*}
\begin{split}
\sum_{j=0}^{\frac{n-1}{2}}m_j(\theta_j-1)&=m_0(\theta_0-1)+\sum_{j=1}^{\frac{n-1}{2}}m_j(\theta_j-1)=1+\displaystyle\sum_{j=1}^{\frac{n-1}{2}}(-1)^{j-1}\left(1-2\cos\left(\frac{j\pi}{n}\right)\right)\\
&=1+\displaystyle\sum_{j=1}^{\frac{n-1}{2}}(-1)^{j-1}
-2\displaystyle\sum_{j=1}^{\frac{n-1}{2}}(-1)^{j-1}\cos\left(\frac{j\pi}{n}\right)=1+0-2(1/2)=0.
\end{split}
\end{equation*}
Applying Theorem \ref{th1}(4), we get no
LPGST between $(0,u)$ and $(1,u)$ in $\up{2}P_n$.

Now, suppose $u\notin\{1,n\}$. Then $d_u=2\in\sigma_u(P_n)$ if and only if $j=\frac{n}{2}$. Suppose $n$ is even. Since $2u-1$ is odd for all $u\notin\{1,n\}$, we have $2n\nmid (2u-1)\frac{n}{2}$, and so by (\ref{supp}), we have $2\in \sigma_u(P_n)$ for all $u\notin\{1,n\}$. Invoking Theorem \ref{sc}(2), we get that $(0,u)$ and $(1,u)$ are not strongly cospectral in $\up{2}P_n$, and hence they cannot admit LPGST. Now, suppose $n$ is odd, so that $2\notin \sigma_u(P_n)$. We proceed with two subcases.

\vspace{0.1in}

\noindent \textbf{Case 1.} Let $u\neq \frac{n+1}{2}$. If $n$ is an odd prime, then $\sigma_u(P_n)=\sigma(P_n)$. In this case, the proof of the case when $u\in\{1,n\}$ applies. Hence, there is no LPGST between $(0,u)$ and $(1,u)$ in $\up{2}P_n$. Now, suppose $n$ is an odd composite number. We adapt the proof of Case 3 in \cite[Theorem 4.4.3]{Bommel2019}. Since $u\neq \frac{n+1}{2}$, there exists a prime factor $p$ of $n$ that divides $\frac{2n}{\operatorname{gcd}(2u-1,2n)}$. Moreover, if $k\not \equiv 0$ (mod $p$), then $\theta_k\in\sigma_u(P_n)$. Now, let $m_j=1$ if $j\equiv 1,p+2$ (mod $2p$), $m_j=-1$ if $j\equiv 2,p+1$ (mod $2p$) and $m_j=0$ otherwise. For $c\in\{1,2\}$, we get
\begin{equation*}
\sum_{j=0}^{n-1}m_j(\theta_j-2)=\sum_{{\ell}=0}^{\frac{n}{p}-1}(-1)^{\ell}(\theta_{c+\ell p}-2)=\sum_{{\ell}=0}^{\frac{n}{p}-1}(-1)^{\ell}\theta_{c+\ell p}-2\sum_{{\ell}=0}^{\frac{n}{p}-1}(-1)^{\ell}=2-2(1)=0.
\end{equation*}
Meanwhile, $\sum_{j=0}^{n-1}m_j=\sum_{{\ell}=0}^{\frac{r}{p}-2}(-1)^{\ell}=1$ is odd. Applying Theorem \ref{th1}(4), we get no
LPGST between $(0,u)$ and $(1,u)$ in $\up{2}P_n$.

\vspace{0.1in}

\noindent \textbf{Case 2.} Let $u=\frac{n+1}{2}$. Then $\sigma_u(P_n)=\{0\}\cup\{\theta_j:j\ \text{odd}\}$. Since $-\cos\left(\frac{j\pi}{n}\right)=\cos\left(\frac{(n-j)\pi}{n}\right)$ for all even $j\in\{1,\ldots,\frac{n-1}{2}\}$, (\ref{eqq}) gives us
\begin{equation}
\label{eq11}
\sum_{j\geq 1\ \text{odd}}^{n-1}\cos\left(\frac{j\pi}{n}\right)=\frac{1}{2}.
\end{equation}
Using (\ref{eq11}), we can write
\begin{equation*}
\sum_{j\geq 1\ \text{odd}}^{n-1}(\theta_{j}-2)=-2\sum_{j\geq 1\ \text{odd}}^{n-1}\cos\left(\frac{j\pi}{n}\right)=-1.
\end{equation*}
Now, let $m_j=2$ for all odd $j\in\{1,\ldots,n-1\}$ and $m_0=-1$. The above equation yields
\begin{equation*}
\begin{split}
\sum_{j=0}^{n-1}m_j(\theta_j-2)&=m_0(\theta_0-2)+\sum_{j\geq 1\ \text{odd}}^{n-1}m_j(\theta_j-2)=-1(-2)+2\sum_{j\geq 1\ \text{odd}}^{n-1}(\theta_{j}-2)=0.
\end{split}
\end{equation*}
Meanwhile, $\sum_{j=0}^{n-1}m_j=-1+2(\frac{n-1}{2})$ is odd. Applying Theorem \ref{th1}(4), we get no LPGST between $(0,u)$ and $(1,u)$ in $\up{2}P_n$.

Combining all cases above yields the conclusion.
\end{proof}

\section{Double stars}
A \textit{double star} graph $S_{k,l}$ is a tree resulting from attaching $k$ and $l$ pendent vertices to the vertices of $K_2$. Our goal is to characterize LPGST in blow-ups of double stars.

The characteristic polynomial of $S_{k,l}$ can be obtained using \cite[Proposition 1]{grone1990ordering} as \[\Phi(x)=x(x-1)^{k+l-2}p(x),\] where $p(x)=x(x-k-1)(x-l-1)-(x-1)(x-k-1)-(x-1)(x-l-1)$.

We now characterize LPGST in blow-ups of double stars.

\begin{thm}
\label{ds}
The following hold.
\begin{enumerate}
\item Let $v$ be a vertex in $S_{k,l}$ of degree $k+1$. There is LPGST in $\up{2}S_{k,l}$ between $(0,v)$ and $(1,v)$ if and only if $p(x)$ is irreducible over $\Ql$ and $\nu_2(k+1)=\nu_2(l+3)$.
\item Let $w$ be a leaf in $S_{k,l}$ attached to a vertex of degree $l+1$.
%Let $k>1$ and $w$ be the unique pendant vertex in $S_{k,1}$ attached to the vertex of degree $2$.
There is LPGST in $\up{2}S_{k,l}$ between $(0,w)$ and $(1,w)$ if and only if $l=1$, $k$ is odd, and either $p(x)$ is irreducible over $\Ql$, or $p(x)$ has a root that is an even integer.
\end{enumerate}
\end{thm}

\begin{proof}
We prove (1). First, suppose $k=l$. In this case, $p(x)$ has $x-(k+1)$ as a linear factor. It is easy to check that $\up{2}S_{1,1}$, the blowup of $P_4,$ does not exhibit LPGST. For $k\geq 2$, we have $d_u=1$ if $u$ is a leaf, while $d_u=k+1$ otherwise. Now, 1 and $k+1$ are eigenvalues of $L(S_{k,k})$ with eigenvectors $\e_u-\e_v$ and $\vl$, respectively, where $u$ and $v$ are leaves attached to the same vertex, and $\vl\e_w^T=1$ whenever $w$ is a leaf and $\vl\e_w^T=-k$ otherwise. Thus, for all $k\geq 2,$
$d_u\in\sigma_u\ob{S_{k,k}}$ for every vertex $u$ of $S_{k,k}$, and so there are no strongly cospectral pairs of vertices in $\up{2}S_{k,k}$ by Theorem \ref{sc}(2). Next, consider the case $k\neq l$. Then $k+1$ and $l+1$ are not eigenvalues of $S_{k,l}$, and all eigenvalues except $1$ belong to the support of the vertex $v$ of degree $k+1$. Thus, vertices $(0,v)$ and $(1,v)$ are strongly cospectral in $\up{2}S_{k,l}$, and so we may apply Theorem \ref{th1}. Since the sum of eigenvalues of $L(X)$ is equal to its trace, we obtain $(k+l-2)(1)+\lambda_1+\lambda_2+\lambda_3=2(k+l+1)$. Equivalently,
$$\lambda_1+\lambda_2+\lambda_3=k+l+4.$$ The assumption in Theorem \ref{th1}(4) holds if and only if $$-m_0(k+1)+m_1(\lambda_1-(k+1))+m_2(\lambda_2-(k+1))+m_3(\lambda_3-(k+1))=0$$ for some integers $m_0,m_1,m_2,m_3$. Combining the above two equations gives us
\begin{equation}
\label{pgstskl}
(m_2-m_1)\lambda_2+(m_3-m_1)\lambda_3=(m_0+m_1+m_2+m_3)(k+1)-m_1(k+l+4)
\end{equation}
We now proceed with two cases: when $p(x)$ has exactly one integer eigenvalue and when $p(x)$ has no integer eigenvalue (equivalently, it is irreducible). Note that $p(x)$ cannot have all integer eigenvalues, for this would imply  that $S_{k,\ell}$ is Laplacian integral, a contradiction to the fact that the stars are the only family of trees that are Laplacian integral.
\\ \textbf{Case 1}. Suppose $p(x)$ irreducible. If $m_2-m_1\neq 0$ but $m_3-m_1=0$, then equation (\ref{pgstskl}) implies that $\lambda_2$ is an integer, a contradiction. Similarly for the case $m_2-m_1=0$ but $m_3-m_1\neq 0$. Now, if $m_2-m_1\neq 0$ and $m_3-m_1\neq 0$, then equation (\ref{pgstskl}) implies that $\lambda_2=c\lambda_3+d$ for some rational numbers $c,d$ with $c\neq 0$. Thus, $\lambda_3$ is a root of $p(x)$ and $p(cx+d)$, which are both monic cubic polynomials with integer coefficients. Since $\lambda_3$ is algebraic over $\Ql$ and its minimal polynomial over $\Ql$ is unique, it follows that $p(x)=p(cx+d)$. That is, $c=1$ and $d=0$. But this implies that $\lambda_2=\lambda_3$, a contradiction to the fact that $p(x)$ is irreducible. The only case left is when $m_2-m_1=0$ and $m_3-m_1=0$. This yields $m_1=m_2=m_3$, and so we may rewrite equation (\ref{pgstskl}) as 
\begin{equation}
\label{eq}
m_0(k+1)/g=m_1(l-2k+1)/g,
\end{equation}
where $g=\operatorname{gcd}(k+1,l-2k+1)=\operatorname{gcd}(k+1,l+3)$. As $m_1=m_2=m_3$, we get that $m_0+m_1+m_2+m_3$ is even if and only if $m_0+m_1$ is even. If $\nu_2(k+1)=\nu_2(l+3)$, then the $m_0$ and $m_1$ that satisfies equation (\ref{eq}) must have the same parity, in which case $m_0+m_1$ is  even. Thus, we get LPGST between $(0,v)$ and $(1,v)$ in $\up{2}S_{k,l}$ by Theorem \ref{th1}(4). However, if $\nu_2(k+1)>\nu_2(l+3)$, then we may choose $m_0=\frac{l-2k+1}{g}$ and $m_1=\frac{k+1}{g}$, then $m_0$ is odd while $m_1$ is even, and so $m_0+m_1$ is odd. In this case, we do not get LPGST between $(0,v)$ and $(1,v)$ in $\up{2}S_{k,l}$ by Theorem \ref{th1}(4). We may argue similarly for the case $\nu_2(k+1)<\nu_2(l+3)$.\\ \textbf{Case 2}. Suppose $p(x)$ has exactly one integer eigenvalue, say $\lambda_2$. Then equation (\ref{pgstskl}) implies that  $m_1=m_3$ (otherwise, $\lambda_3$ is also an integer, a contradiction). Thus, $m_0+m_1+m_2+m_3$ is even if and only if $m_0+m_2$ is even. Moreover, we may rewrite equation (\ref{pgstskl}) as
\begin{equation}
\label{eq1}
(m_2-m_1)\lambda_2=(m_0+m_1+m_2)(k+1)-m_1(l+3).
\end{equation}
First, suppose $\nu_2(k+1)\neq \nu_2(l+3)$.
Set $m_1=m_2$ so that $m_1\frac{l+3}{g}=(m_0+2m_1)\frac{k+1}{g}$ by equation (\ref{eq1}), where $g=\operatorname{gcd}(k+1,l+3)$. Since $\nu_2(k+1)\neq \nu_2(l+3)$, letting $m_0=\frac{l+3-2(k+1)}{g}$ and $m_1=\frac{k+1}{g}$, we get that $m_0$ and $m_1$ satisfy equation (\ref{eq1}) and $m_0+m_1=m_0+m_2$ is odd. Now, suppose $\nu_2(l+3)=\nu_2(k+1)$. We have two subcases. Suppose $\nu_2(\lambda_2)\leq \nu_2(k+1)$. Setting $m_1=0$, we obtain $m_2\frac{\lambda_2}{g}=(m_0+m_2)\frac{k+1}{g}$ from equation (\ref{eq1}), where $g=\operatorname{gcd}(\lambda_2,k+1)$. Letting 
$m_0=\frac{\lambda_2-k-1}{g}$ and $m_2=\frac{k+1}{g}$, we get that $m_0$ and $m_2$ satisfy equation (\ref{eq1}) and $m_0+m_2$ is odd. For the case that $\nu_2(\lambda_2)>\nu_2(k+1)$, set $m_2=0$ so that $m_1\frac{l+3-(k+1)-\lambda_2}{g}=m_0\frac{k+1}{g}$ by equation (\ref{eq1}), where $g=\operatorname{gcd}(\lambda_2,k+1,l+3)$. Letting 
$m_0=\frac{l+3-(k+1)-\lambda_2}{g}$ and $m_1=\frac{k+1}{g}$, we get that $m_0$ and $m_1$ satisfy equation (\ref{eq1}) and $m_0$ is odd, and so $m_0+m_2$ is odd. In any case, we see that we can find integers $m_0,m_1,m_2,m_3$ satisfying the assumption of Theorem \ref{th1}(4) whose sum is odd. Hence, there is no LPGST between $(0,v)$ and $(1,v)$ in $\up{2}S_{k,l}$.
Combining the two cases above establishes 1.

We now prove 2. If $l>1$, then a leaf $w$ attached to a vertex of degree $l+1$ has at least one twin $u$. Thus, vertices $(0,w)$, $(1,w)$, $(0,u)$ and $(1,u)$ form a twin set in $\up{2}S_{k,l}$. Invoking \cite[Corollary 3.10]{mon1}, these vertices do not admit strong cospectrality, and hence fail to exhibit LPGST. Thus, $l=1$. In this case, all eigenvalues except $1$ belong to the support of the vertex $w.$ Thus, vertices $(0,w)$ and $(1,w)$ are strongly cospectral in $\up{2}S_{k,l}$ by Theorem \ref{sc}(2), and so we again apply Theorem \ref{th1}. Arguing similarly as above yields
\begin{equation}
\label{sk1}
(m_2-m_1)\lambda_2+(m_3-m_1)\lambda_3=(m_0+m_1+m_2+m_3)-m_1(k+5)
\end{equation}
We again proceed with two cases. First, suppose $p(x)$ irreducible. The same argument above yields $m_1=m_2=m_3$. Thus, $m_0+m_1+m_2+m_3$ is even if and only if $m_0+m_1$ is even. From equation (\ref{sk1}), we have $m_0=m_1(k+2).$ If $k$ is odd, then $m_0$ and $m_1$ have the same parity, so $m_0+m_1$ is even. Thus, we get LPGST between $(0,w)$ and $(1,w)$ in $\up{2}S_{k,l}$ by Theorem \ref{th1}(4). But if $k$ is even, then we may let $m_0=k+2$ and $m_1=1$, then $m_0+m_1$ is odd so we do not get LPGST between $(0,w)$ and $(1,w)$ in $\up{2}S_{k,l}$ by Theorem \ref{th1}(4). Now, suppose $p(x)$ has exactly one integer eigenvalue, say $\lambda_2$. Using the same argument above, we may rewrite equation (\ref{sk1}) as
\begin{equation}
\label{sleq1}
m_0+m_2=(m_2-m_1)\lambda_2+m_1(k+3).
\end{equation}
First, suppose $k$ is even. Set $m_1=m_2=1$ and $m_0=k+2$, we get that $m_0$ and $m_1$ satisfy equation (\ref{sleq1}) and $m_0+m_2$ is odd. In this case, we do not get LPGST. Now, suppose $k$ is odd. We have two subcases. If $\lambda_2$ is odd, then $m_1=0$ and $m_2=1$ satisfies equation (\ref{sleq1}) but $m_0+m_2$ is odd. If $\lambda_2$ is even, then $m_0+m_2$ is always even as $k$ is odd. Hence, LPGST occurs in this case between $(0,w)$ and $(1,w)$ in $\up{2}S_{k,1}$ by Theorem \ref{th1}(4). Combining these two cases proves 2.
\end{proof}

\begin{rem}
If $u$ is a vertex of $S_{k,l}$ of degree $l+1$, then a criterion for LPGST between $(0,u)$ and $(1,u)$ in $\up{2}S_{k,l}$ is obtained by switching the roles of $k$ and $l$ in Theorem \ref{ds}(1). A similar observation holds for the case when $w$ is a leaf in $S_{k,l}$ attached to a vertex of degree $k+1$.
\end{rem}

\section{Edge perturbation}\label{sec:pert}

In Corollary \ref{hc1}, we have observed that if $v$ is a vertex in $G$ with $\sigma_v(G)\subset\Zl$ then $\up{n}G$ is periodic at the vertex $(j,v)$ for all $j$ and $n$ with period $\frac{2\pi}{n}.$ It is clear that if $n\equiv 0$ (mod 4) then $\up{n}G$ is periodic at the vertex $(j,v)$ for all $j$ and $n$ with period $\frac{\pi}{2}.$  In this case, we have $n\geq 3$, and so $\up{n}G$ does not exhibit strong cospectrality by Theorem \ref{sc}(1). Consequently, $\up{n}G$ cannot exhibit LPST. However, using the fact that $T_v=\cb{(j,v) \mid j\in \Zl_n}$ consists of pairwise false twins in $\up{n}G$, the insertion of an additional edge in $\up{n}G$ between vertices $(l,v),(m,v)\in T_v$ with $l\neq m$ results in LPST between $(l,v)$ and $(m,v)$ at $\frac{\pi}{2}$ \cite[Theorem 4]{pal8}. Moreover, the resulting graph is periodic at the remaining vertices in the set $T=\cb{(j,v) \mid  j\neq l,m}$ with period $\frac{\pi}{2}.$ Invoking \cite[Theorem 3]{pal8}, another edge can again be added between a pair of vertices in $T$ to have LPST $\frac{\pi}{2}.$ Applying this process inductively, we may generate new graphs exhibiting LPST by simply adding a \textit{matching} in $T_v$, which is a set of edges without common vertices. We summarize this as follows.

\begin{thm}\label{pt1}
Let $n\equiv 0$ (mod 4) and $v$ be a vertex in $G$ with $\sigma_v(G)\subset\Zl.$ Then the addition of a matching in $T_v$ results in LPST in $\up{n}G$ at $\frac{\pi}{2}$ between the end vertices of every edge inserted.  Moreover, the resulting graph is periodic at all vertices in $T_v$ that are not incident to the newly added edges with period $\frac{\pi}{2}.$ 
\end{thm}

According to \cite[Theorem 3]{pal8}, the addition of a matching in $\up{n}G$ as described above works as long as the vertices in the matching belong to a set of twins. Thus, if $v$ and $u$ are false twins, then $T_v\cup T_u$ is a twin set by Proposition \ref{twins}, and so the conclusion in Theorem \ref{pt1} still works if we add a matching in $T_v\cup T_u$. Thus, the following is immediate.

\begin{cor}\label{pt2}
Let $n\equiv 0$ (mod 4). If $u$ and $v$ are false twins in $G$ with $\sigma_v(G)\subset\Zl$, then the addition of a matching in $T_v\cup T_u$ results in LPST in $\up{n}G$ at $\frac{\pi}{2}$ between the end vertices of every edge inserted. Moreover, the resulting graph is periodic at all vertices in $T_v$ that are not incident to the newly added edges with period $\frac{\pi}{2}.$ 
\end{cor}

In general, we may be begin with any Laplacian integral graph $G$, and apply \cite[Theorem 3]{pal8} and \cite[Theorem 4]{pal8} successively to $\up{n}G$ where $n\equiv 0$ (mod 4) to construct new graphs having LPST. 

Combining Theorem \ref{pt1} and Corollary \ref{pt2} yields the following result.

\begin{cor}\label{pt3}
Let $n\equiv 0$ (mod 4) and $v$ be a vertex in $G$ with $\sigma_v(G)\subset\Zl.$
\begin{enumerate}
\item If $v$ has no false twins in $G$, then the addition of a perfect matching in $T_v$ results in LPST in $\up{n}G$ at $\frac{\pi}{2}$ between the end vertices of every edge inserted.
\item If $v$ has a false twin $u$ in $G$, then the addition of a perfect matching in $T_v\cup T_u$ results in LPST in $\up{n}G$ at $\frac{\pi}{2}$ between the end vertices of every edge inserted. 
\end{enumerate}
\end{cor}

We demonstrate Theorem \ref{pt1} and Corollary \ref{pt3}(1) using the following example.

\begin{figure}[h]
\centering
\begin{tikzpicture}[scale=0.3,auto=left]
\tikzstyle{every node}=[draw, circle, thick, fill=black!0, scale=0.25]

 \node (0) at (0,5.5) {\LARGE$(0,u)$};
 \node (1) at (4,8) {\LARGE$(1,u)$};
 \node (2) at (8,8) {\LARGE$(2,u)$};
 \node (3) at (12,5.5) 
 {\LARGE$(3,u)$};

 \node (4) at (-4,-1) {\LARGE$(0,v)$};
 \node (5) at (-4,-5) {\LARGE$(1,v)$};
 \node (6) at (-2,-8) {\LARGE$(2,v)$};
 \node (7) at (2,-10) {\LARGE$(3,v)$};

 \node (8) at (16,-1) {\LARGE$(0,w)$};
 \node (9) at (16,-5) {\LARGE$(1,w)$};
 \node (10) at (14,-8) {\LARGE$(2,w)$};
 \node (11) at (10,-10) {\LARGE$(3,w)$};

  \foreach \x in {4,5,6,7,8,9,10,11}{
   \draw[thick,black!70] (0)--(\x)--(1);
  }
  \foreach \x in {4,5,6,7,8,9,10,11}{
   \draw[thick,black!70] (2)--(\x)--(3);
  }
  \foreach \x in {8,9,10,11}{
   \draw[thick,black!70] (4)--(\x)--(5);
  }
  \foreach \x in {8,9,10,11}{
   \draw[thick,black!70] (6)--(\x)--(7);
  }

  \draw[thick,dashed] (0)--(2);
  \draw[thick,dashed] (1)--(3);
  \draw[thick,dashed] (4)--(5);
  \draw[thick,dashed] (6)--(7);
  \draw[thick,dashed] (9)--(10);
  \draw[thick,dashed] (8)--(11);
  
 \end{tikzpicture}
 \caption{LPST in $\up{4}K_3$ between the end vertices of the dashed edges.}
 \label{lpco}
 \end{figure}
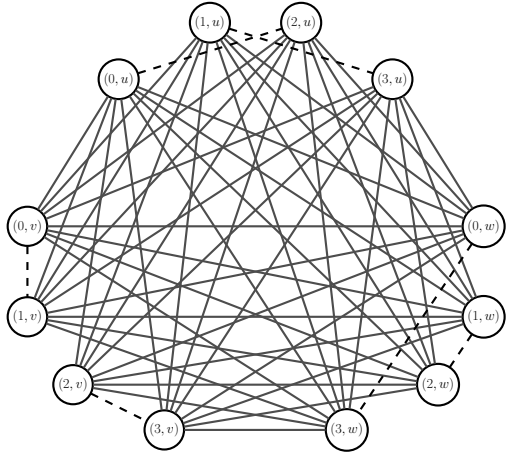

\begin{exm}
Consider $G=K_3,$  on  vertices $u,v,w$. Here $\up{4}G$ can be realised as a complete $3$-partite graph as shown in Figure \ref{lpco} with partite sets $T_u,T_v$ and $T_w.$  As $G$ is an integral graph, addition of any set of pairwise non-adjacent edges in $T_u,T_v$ and $T_w$ results LPST at $\frac{\pi}{2}$ between the end vertices of every edge inserted.    
\end{exm}

We end this section with an illustration of Theorem \ref{pt1} and Corollary \ref{pt3}(2).

\begin{exm}
Consider $G=C_4,$ on vertices $\{a, b, c, d\}$ with edges $\{a,b\},\{b,c\},\{c,d\},\{d,a\}$. Note that $\{a,c\}$ and $\{b,d\}$ are pairs of false twins in $C_4$. As $G$ is an integral graph, addition of any set of pairwise non-adjacent edges in $T_a\cup T_c$ and $T_b\cup T_d$ results LPST at $\frac{\pi}{2}$ between the end vertices of every edge inserted, see Figure \ref{lpco1}.   
\end{exm}

\begin{figure}[h]
\centering
\begin{tikzpicture}[scale=.35,auto=left]
\tikzstyle{every node}=[draw, circle, thick, fill=black!0, scale=0.25]

 \node (0) at (-5,0.5) {\LARGE$(0,a)$};
 \node (1) at (-3,2.5) {\LARGE$(1,a)$};
 \node (2) at (-1,4.5) {\LARGE$(2,a)$};
 \node (3) at (1,6.5) {\LARGE$(3,a)$};
 
 \node (4) at (-5,-5) {\LARGE$(0,b)$};
 \node (5) at (-3,-7) {\LARGE$(1,b)$};
 \node (6) at (-1,-9) {\LARGE$(2,b)$};
 \node (7) at (1,-11) {\LARGE$(3,b)$};

 \node (8) at (13,-5) {\LARGE$(0,c)$};
 \node (9) at (11,-7) {\LARGE$(1,c)$};
 \node (10) at (9,-9) {\LARGE$(2,c)$};
 \node (11) at (7,-11) {\LARGE$(3,c)$};

  \node (12) at (13,0.5) {\LARGE$(0,d)$};
 \node (13) at (11,2.5) {\LARGE$(1,d)$};
 \node (14) at (9,4.5) {\LARGE$(2,d)$};
 \node (15) at (7,6.5) {\LARGE$(3,d)$};

  \foreach \x in {4,5,6,7,12,13,14,15}{
   \draw[thick,black!70] (0)--(\x)--(1);
  }
  \foreach \x in {4,5,6,7,12,13,14,15}{
   \draw[thick,black!70] (2)--(\x)--(3);
  }
  \foreach \x in {8,9,10,11}{
   \draw[thick,black!70] (4)--(\x)--(5);
  }
  \foreach \x in {8,9,10,11}{
   \draw[thick,black!70] (6)--(\x)--(7);
  }
  \foreach \x in {12,13,14,15}{
   \draw[thick,black!70] (8)--(\x)--(9);
  }
  \foreach \x in {12,13,14,15}{
   \draw[thick,black!70] (10)--(\x)--(11);
  }

  \draw[thick,dashed] (4)--(15);
  \draw[thick,dashed] (5)--(14);
  \draw[thick,dashed] (6)--(13);
  \draw[thick,dashed] (7)--(12);
  \draw[thick,dashed] (0)--(11);
  \draw[thick,dashed] (1)--(10);
  \draw[thick,dashed] (2)--(9);
  \draw[thick,dashed] (3)--(8);
  
 \end{tikzpicture}
 \caption{LPST in $\up{4}C_4$ between the end vertices of the dashed edges.}
 \label{lpco1}
 \end{figure}
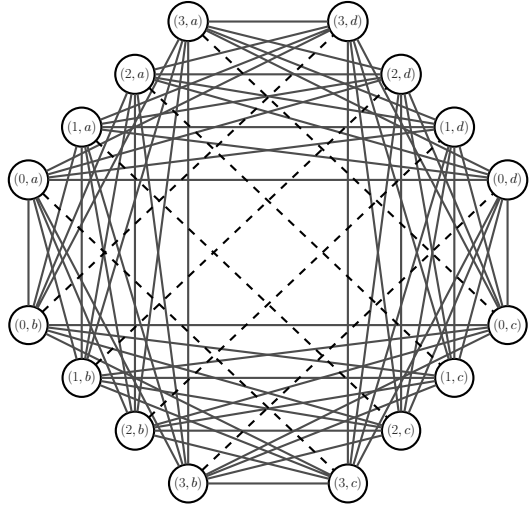

\section{Future work}\label{sec:fw}

We investigated Laplacian quantum walks on blow-up graphs. We characterized periodicity, strong cospectrality, LPST and LPGST in blow-up graphs. Several constructions of blow-up graphs with LPST were also presented. This produced infinite families of regular blow-up graphs  each vertex of which is involved in LPST, though the underlying graphs do not admit LPST. To inspire further work in this topic, we pose the following questions.

A graph $G$ is \textit{weakly Hadamard diagonalizable} (WHD) if its Laplacian matrix is diagonalizable by a weak Hadamard matrix, i.e.  a matrix $H$ with entries from the set $\{0,\pm 1\}$ such that $H^TH$ is tridiagonal. LPST was studied in WHD graphs \cite{mclaren2023weak}.
In line with our result in Corollary \ref{hd}, we ask: when does a blow-up of a WHD graph admit LPST?

Lastly, we ask: does Theorem \ref{pt1} hold for the adjacency case? That is, if $n\equiv 0$ (mod 4), will the addition of an appropriate matching in $\up{n}G$ induce adjacency LPST?

\bigskip

\noindent \textbf{Acknowledgement.}
H.\ Monterde is supported by the University of Manitoba Faculty of Science and Faculty of Graduate Studies. S.\ Kirkland is supported by NSERC grant number RGPIN-2025-05547.

%%%%%%% THE BIBLIOGRAPHY %%%%%%%
\bibliographystyle{abbrv}
\bibliography{References}
\end{document}